\documentclass[12pt]{article}
\usepackage{amsmath,amssymb,amsthm}
\usepackage{amsfonts, amsmath}
\usepackage[english]{babel}
\newcommand{\R}{\mathbb R}

\newcommand{\ep}{\varepsilon}
\DeclareMathOperator{\divg}{div}

\def\essinf{\mathop {\rm ess\,inf}}

\newtheorem{theorem}{Theorem}[section]
\newtheorem{lemma}{Lemma}[section]
\theoremstyle{remark}
\newtheorem{remark}{Remark}[section]
\newtheorem{proposition}{Proposition}[section]

\numberwithin{equation}{section}

\begin{document}
\title{\bf On a Robin problem with $p$-Laplacian and reaction bounded only from above}
\author{
\bf Salvatore A. Marano\thanks{Corresponding author}\\
\small{Dipartimento di Matematica e Informatica,
Universit\`a degli Studi di Catania,}\\
\small{Viale A. Doria 6, 95125 Catania, Italy}\\
\small{\it E-mail: marano@dmi.unict.it}\\
\mbox{}\\
\bf Nikolaos S. Papageorgiou\\
\small{Department of Mathematics,
National Technical University of Athens,}\\
\small{Zografou Campus, Athens 15780, Greece}\\
\small{\it E-mail:npapg@math.ntua.gr}
}
\date{}
\maketitle
\begin{abstract}
The existence of three smooth solutions, one negative, one positive, and one nodal, to a homogeneous Robin problem with $p$-Laplacian and Carath\'eodory reaction is established. No sub-critical growth condition is taken on. Proofs exploit variational as well as truncation techniques. The case $p=2$ is separately examined, obtaining a further nodal solution via Morse's theory.
\end{abstract}
\vspace{2ex}
\noindent {\bf Keywords:} p-Laplacian, Robin problem, constant-sign solution, nodal solution

\vspace{2ex}

\noindent {\bf AMS Subject Classification:} 35J20, 35J60, 35J92
\section{Introduction} \label{S1}
Let $\Omega$ be a bounded domain in $\R^N$ with a smooth boundary $\partial\Omega$, let $1<p<\infty$, let $f:\Omega\times\R \to\R$ be a Carath\'eodory function, and let $\beta\in C^{0,\alpha}(\partial\Omega,\R^+_0)$ for some $\alpha\in (0,1)$. Consider the homogeneous Robin problem
\begin{equation}\label{prob}
\left\{
\begin{array}{ll}
-\Delta_p u=f(x,u) & \mbox{in }\Omega,\\
\phantom{}\\
\displaystyle{\frac{\partial u}{\partial n_p}}+\beta(x)|u|^{p-2}u=0 & \mbox{on } \partial\Omega,\\
\end{array}
\right.
\end{equation}
where $\Delta_p$ denotes the $p$-Laplace differential operator, namely $\Delta_p u:=\divg(|\nabla u|^{p-2}\nabla u)$ for all $u\in W^{1,p}(\Omega)$, while $\frac{\partial u}{\partial n_p}:=|\nabla u|^{p-2}\nabla u\cdot n$, with $n(x)$ being the outward unit normal vector to $\partial\Omega$ at its point $x$. As in \cite[p. 1066]{Le},  $u\in W^{1,p}(\Omega)$ is called a (weak) solution of \eqref{prob} provided
$$\int_\Omega|\nabla u|^{p-2}\nabla u\cdot\nabla v\, dx+\int_{\partial\Omega}\beta|u|^{p-2}uv\, d\sigma =\int_\Omega f(x,u)v\, dx\quad\forall\, v\in W^{1,p}(\Omega).$$
Equations driven by $p$-Laplacian type operators have been widely investigated under Dirichlet boundary conditions, mainly through variational, sub-super-solutions, and truncation techniques \cite{MoMoPa,CH,CLM}, besides Morse's theory \cite{PeSc}. There is a wealth of good results and the relevant literature looks daily increasing. On the other hand, these methods cannot always be adapted in a simple way to treat Neumann (i.e., $\beta\equiv 0$), or more generally Robin ($\beta\not\equiv 0$), problems. That's why over the last few years the study of \eqref{prob} has been receiving attention and very nice papers are already available. The more close to our work are \cite{BaCaMo,GaPaCPAA,MoTa} and, above all, \cite{PaRa}. Indeed, here, we prove the existence of three $C^1$-solutions to Problem \eqref{prob}, one positive, one negative, and one nodal, without assuming that $t\mapsto f(x,t)$ exhibits a sub-critical behavior but is merely bounded on bounded sets. Moreover, roughly speaking, we suppose that
$$\limsup_{t\to\pm\infty}\frac{f(x,t)}{|t|^{p-2}t}\leq a_0<\lambda_1\;\;\mbox{and}\;\;\lambda_2<a_1\leq\liminf_{t\to 0}\frac{f(x,t)}{|t|^{p-2}t}\leq\limsup_{t\to 0}\frac{f(x,t)}{|t|^{p-2}t}\leq a_2<+\infty$$
uniformly in $x\in\Omega$, with $\lambda_1$ (respectively, $\lambda_2$) being the first (respectively, second) eigenvalue of $(-\Delta_p, W^{1,p}(\Omega))$ under Robin's boundary condition; see Section 3 for precise formulations. So, no global growth from below is imposed on $t\mapsto f(x,t)$. The meaningful special case 
$$f(x,t):=\lambda |t|^{p-2}t-g(x,t).\quad (x,t)\in\Omega\times\mathbb{R},$$
where $\lambda>\lambda_2$, is also examined and some results of \cite{PaRa} extended; cf. also \cite{BaCaMo,GaPaCPAA, MoTa}, which however require $\beta\equiv 0$. When $p=2$ we obtain a second nodal solution by assuming, among other things, $f(x,\cdot)\in C^1(\mathbb{R})$ and
$$|f'_t(x,t)|\leq a_3(1+|t|^{r-2})\quad\forall\, (x,t)\in\Omega\times\mathbb{R},$$
with $2\leq r<2^*$. Let us finally point out that an analogous investigation might be performed for the problem
$$-\Delta_p u+a(x)|u|^{p-2}u=f(x,u)\quad\mbox{in}\quad\Omega,\quad\frac{\partial u}{\partial n_p}=0\quad\mbox{on}\quad \partial\Omega,$$
where $a\in L^\infty(\Omega)$ may change sign, exploiting the results of \cite{MuPa}. 
\section{Preliminaries}\label{S2}
Let $(X,\Vert\cdot\Vert)$ be a real Banach space. Given a set $V\subseteq X$, write $\overline{V}$ for the closure of $V$, $\partial V$ for the boundary of $V$, and ${\rm int}_X(V)$ or simply ${\rm int}(V)$, when no confusion can arise, for the interior of $V$. If $x\in X$ and $\delta>0$ then
$$B_\delta(x):=\{ z\in X:\;\Vert z-x\Vert<\delta\}\, .$$
The symbol $(X^*,\Vert\cdot \Vert_{X^*})$ denotes the dual space of $X$, $\langle\cdot,\cdot\rangle$ indicates the duality pairing between $X$ and $X^*$, while $x_n\to x$ (respectively, $x_n\rightharpoonup x$) in $X$ means `the sequence $\{x_n\}$ converges strongly (respectively, weakly) in $X$'.

Let $T$ be a topological space and let $L$ be a multifunction from $T$ into $X$ (briefly, $L:T\to  2^X$), namely a function which assigns to each $t\in T$ a nonempty subset $L(t)$ of $X$. We say that $L$ is lower semi-continuous when $\{t\in T: L(t)\cap V\neq\emptyset\}$ turns out to be open in $T$ for every open set $V\subseteq X$. A function $l:T\to X$ is called a selection of $L$
provided $l(t)\in L(t)$ for all $t\in T$.

We say that $\Phi:X\to\mathbb{R}$ is coercive iff
$$\lim_{\Vert x\Vert\to+\infty}\Phi(x)=+\infty,$$
while $\Phi$ is called weakly sequentially lower semi-continuous when $x_n\rightharpoonup x$ in $X$ implies $\Phi(x)\leq \displaystyle{\liminf_{n\to\infty}}\Phi(x_n)$. Let $\Phi\in C^1(X)$. The classical Palais-Smale compactness condition for $\Phi$ reads as follows.
\begin{itemize}
\item[$({\rm PS})$] {\it Every sequence $\{x_n\}\subseteq X$ such that $\{\Phi(x_n)\}$ is bounded and $\Vert \Phi'(x_n) \Vert_{X^*}\to 0$ has a convergent subsequence.}
\end{itemize}
Define, provided $c\in\mathbb{R}$,
$$\Phi^c:=\{x\in X:\; \Phi(x)\leq c\}\, ,\quad K_c(\Phi):=K(\Phi)\cap\Phi^{-1}(c)\, ,$$
where, as usual, $K(\Phi)$ denotes the critical set of $\Phi$, i.e., $K(\Phi):=\{x\in X:\,\Phi'(x)=0\}$. 

We say that $A:X\to X^*$ is of type $({\rm S})_+$ iff
$$x_n\rightharpoonup x\quad\mbox{in X,}\quad
\limsup_{n\to+\infty}\langle A(x_n),x_n-x\rangle\leq 0$$
imply $x_n\to x$. The next elementary result \cite[Proposition 2.2]{MaPaPEMS} will be employed later. 
\begin{proposition}\label{compactness}
Let $X$ be reflexive and let $\Phi\in C^1(X)$ be coercive. Assume $\Phi'=A+B$, with $A:X\to X^*$ of type $({\rm S})_+$ and $B:X\to X^*$ compact. Then $\Phi$ satisfies $({\rm PS})$.
\end{proposition}
Given a topological pair $(A,B)$ fulfilling $B\subset A\subseteq X$, the symbol $H_q(A,B)$, $q\in\mathbb{N}_0$, indicates the
${\rm q}^{\rm th}$-relative singular homology group of $(A,B)$ with integer coefficients. If $x_0\in K_c(\Phi)$ is an isolated point of $K(\Phi)$ then
$$C_q(\Phi,x_0):=H_q(\Phi^c\cap V,\Phi^c\cap V\setminus\{x_0\})\, ,\quad q\in
\mathbb{N}_0\, ,$$
are the critical groups of $\Phi$ at $x_0$. Here, $V$ stands for any neighborhood of $x_0$ such that $K(\Phi)\cap\Phi^c\cap V=\{x_0\}$. By excision, this definition does not depend on the choice of $V$. Suppose $\Phi$ satisfies Condition $({\rm PS})$, $\Phi|_{K(\Phi)}$ is bounded below, and $c<\displaystyle{\inf_{x\in K(\Phi)}}\Phi(x)$. Put
$$C_q(\Phi,\infty):=H_q(X,\Phi^c)\, ,\quad q\in\mathbb{N}_0\,.$$
The second deformation lemma \cite[Theorem 5.1.33]{GaPaNA} implies that this definition does not depend on the choice of $c$. If $K(\Phi)$ is finite, then setting
$$M(t,x):=\sum_{q=0}^{+\infty}{\rm rank}\, C_q(\Phi,x)t^q\, ,\quad
P(t,\infty):=\sum_{q=0}^{+\infty}{\rm rank}\, C_q(\Phi,\infty)t^q\quad
\forall\, (t,x)\in\mathbb{R}\times K(\Phi)\, ,$$
the following Morse relation holds:
\begin{equation}\label{morse}
\sum_{x\in K(\Phi)}M(t,x)=P(t,\infty)+(1+t)Q(t)\, ,
\end{equation}
where $Q(t)$ denotes a formal series with nonnegative integer coefficients; see for instance \cite[Theorem 6.62]{MoMoPa}. 

Now, let $X$ be a Hilbert space, let $x\in K(\Phi)$, and let $\Phi$ be $C^2$ in a neighborhood of $x$. If $\Phi''(x)$ turns out to be invertible, then $x$ is called non-degenerate. The Morse index $d$ of $x$ is the supremum of the dimensions of the vector subspaces of $X$ on which $\Phi''(x)$ turns out to be negative definite. When $x$ is non-degenerate and with Morse index $d$ one has
\begin{equation}\label{kd}
C_q(\Phi,x)=\delta_{q,d}\mathbb{Z}\, ,\quad q\in\mathbb{N}_0\, .
\end{equation}
The monographs \cite{MW, MoMoPa} represent general references on the subject.

Throughout the paper, $\Omega$ denotes a bounded domain of the real euclidean $N$-space $(\mathbb{R}^N,|\cdot|)$ whose boundary is $C^2$ while $\beta\in C^{0,\alpha}(\partial\Omega,\R^+_0)$ for some $\alpha\in (0,1)$ and $\beta\not\equiv 0$. On $\partial\Omega$ we will employ the $(N-1)$-dimensional Hausdorff measure $\sigma$. The symbol $m$ stands for the Lebesgue measure, $p\in (1,+\infty)$, $p':=p/(p-1)$, $\Vert\cdot\Vert_q$ with $q\geq 1$ indicates the usual norm of $L^q(\Omega)$, $X:=W^{1,p}(\Omega)$, and
$$\Vert u\Vert:=\left(\Vert\nabla u\Vert_p^p+\Vert u\Vert_p^p\right)^{1/p},\quad u\in X,$$
$$C_+:=\{u\in C^0(\overline{\Omega}): u(x)\geq 0\quad\forall\, x\in\overline{\Omega}\}.$$
Write $p^*$ for the critical exponent of the Sobolev embedding $W^{1,p}(\Omega)\subseteq L^q(\Omega)$. Recall that $p^*=Np/(N-p)$ if $p<N$, $p^*=+\infty$ otherwise, and the embedding is compact whenever $1\leq q<p^*$. Moreover,
$${\rm int}(C_+)=\{ u\in C_+: u(x)>0\quad\forall\, x\in\overline{\Omega}\}.$$
Given $t\in\mathbb{R}$, $u,v:\Omega\to\mathbb{R}$, and $f:\Omega\times\mathbb{R}\to\mathbb{R}$, define
$$t^\pm:=\max\{\pm t,0\},\quad u^\pm (x):=u(x)^\pm ,\quad N_f(u)(x):=f(x,u(x)).$$
The symbol $u\leq v$ means $u(x)\leq v(x)$ for almost every $x\in\Omega$. If $u,v$ belong to a function space $Y$ and $u\leq v$ then we set
$$[u,v]:=\{w\in Y:u\leq w\leq v\}.$$
Let $A_p:X\to X^*$ be the nonlinear operator stemming from the negative $p$-Laplacian $\Delta_p$, i.e.,
\begin{equation*}
\langle A_p(u),v\rangle:=\int_\Omega\vert\nabla u(x)\vert^{p-2}\nabla u(x)\cdot\nabla v(x)\, dx\quad\forall\, u,v\in X.
\end{equation*}
A standard argument \cite[Proposition 2.71]{MoMoPa} ensures that $A_p$ is of type $({\rm S})_+$. 
\begin{remark}\label{equiv}
Given $u\in X$ and $w\in L^{p'}(\Omega)$, the condition
$$\langle A_p(u),v\rangle+\int_{\partial\Omega}\beta(x)|u(x)|^{p-2}u(x)v(x)d\sigma=\int_\Omega w(x)v(x)dx,\quad v\in X,$$
is equivalent to  
$$-\Delta_p u=w\quad\mbox{in}\quad\Omega,\quad\frac{\partial u}{\partial n_p}+\beta(x)|u|^{p-2}u=0\quad\mbox{on}\quad
\partial\Omega.$$
This easily comes out from the nonlinear Green's identity \cite[Theorem 2.4.54]{GaPaNA}; see for instance the proof of \cite[Proposition 3]{PaRa}.
\end{remark}

We shall employ some facts on the spectrum $\sigma(-\Delta_p)$ of the operator $-\Delta_p$ with homogeneous Robin boundary conditions. So, consider the nonlinear eigenvalue problem
\begin{equation}\label{eigen}
\left\{
\begin{array}{ll}
-\Delta_p u=\lambda |u|^{p-2}u & \mbox{in }\Omega,\\
\displaystyle{\frac{\partial u}{\partial n_p}}+\beta(x)|u|^{p-2}u=0 & \mbox{on } \partial\Omega.\\
\end{array}
\right.
\end{equation}
The Liusternik-Schnirelman theory provides a strictly increasing sequence $\{\lambda_n\}\subseteq\mathbb{R}^+_0$ of eigenvalues for \eqref{eigen}. Denote by $E(\lambda_n)$ the eigenspace corresponding to $\lambda_n$, $n\in\mathbb{N}$. From \cite{Le,PaRa} we know that:
\begin{itemize}
\item[$({\rm p}_1)$] \textit{$\lambda_1$ is positive, isolated, and simple. Further,}
$$\lambda_1=\inf\left\{ \frac{\Vert\nabla u\Vert^p_p+\int_{\partial\Omega}\beta|u|^pd\sigma}{\Vert u\Vert_p^p}: u\in X,\; u\neq 0\right\}.$$
\item[$({\rm p}_2)$] \textit{There exists an $L^p$-normalized eigenfunction $\hat u_1\in{\rm int}(C_+)$ associated with $\lambda_1$.}
\end{itemize}
The next characterization of $\lambda_2$ will be used later. For its proof we refer the reader to \cite[Proposition 5]{PaRa}.
\begin{itemize}
\item[$({\rm p}_3)$] \textit{Write $U:=\{ u\in X:\; \Vert u\Vert_p=1\}$ as well as
$$\Gamma_1:=\{\gamma\in C^0([-1,1],U):\gamma(-1)=-\hat u_1,\;\gamma(1)=\hat u_1\},$$
$$\Phi(u):=\Vert\nabla u\Vert_p^p+\int_{\partial\Omega}\beta(x)|u(x)|^pd\sigma,\quad u\in X.$$
Then
$$\lambda_2=\inf_{\gamma\in\Gamma_1}\max_{t\in [-1,1]}\Phi(\gamma(t)) .$$
}
\end{itemize}
Define $U_C:=\{u\in C^1(\overline{\Omega}): \|u\|_p=1\}$. Evidently, $U_C$ turns out to be dense in $U$. Let
$$\Gamma_C:=\{\gamma\in C^0([-1,1],U_C):\;\gamma(-1)=-\hat u_1,\,\gamma(1)=\hat u_1\}$$
\begin{lemma}\label{density}
The set $\Gamma_C$ is dense in $\Gamma_1$ with respect to the usual norm of $C^0([-1,1],X)$.
\end{lemma}
\begin{proof}
Pick any $\gamma\in\Gamma_1$. We shall prove that there exists a sequence $\{\gamma_n\}\subseteq\Gamma_C$ fulfilling
\begin{equation}\label{l1}
\lim_{n\to+\infty}\max_{t\in[-1,1]}\Vert\gamma_n(t)-\gamma(t)\Vert=0.
\end{equation}
The multifunction $L_n:[-1,1]\to 2^{C^1(\overline{\Omega})}$ defined by
$$L_n(t):=
\begin{cases}
\{-\hat u_1\} &\mbox{when}\;\; t=-1,\\
\{u\in C^1(\overline{\Omega}):\Vert u-\gamma(t)\Vert<1/n\} &\mbox{if}\;\; t\in (-1,1),\\
\{\hat u_1\} &\mbox{when}\;\; t=1
\end{cases}
$$
takes nonempty convex values and is lower semi-continuous. So, Theorem $3.1'''$ in \cite{Mi}  provides a continuous selection
$l_n:[-1,1]\to C^1(\overline{\Omega})$ of $L_n$. This entails
\begin{equation}\label{l2}
\Vert l_n(t)-\gamma(t)\Vert<\frac{1}{n}\quad\forall\, t\in (-1,1),\quad l_n(-1)=-\hat u_1,\quad l_n(1)=\hat u_1.
\end{equation}
Consequently,
\begin{equation}\label{l3}
\lim_{n\to+\infty}\Vert l_n(t)\Vert_p=\Vert\gamma(t)\Vert_p=1
\end{equation}
uniformly with respect to $t\in[-1,1]$. For any $n$ large enough we can thus set
$$\gamma_n(t):=\frac{l_n(t)}{\Vert l_n(t)\Vert_p},\quad t\in[-1,1].$$
On account of \eqref{l2} and $({\rm p}_3)$ one has $\gamma_n\in\Gamma_C$. Moreover,
\begin{eqnarray}\label{l4}
\Vert\gamma_n(t)-\gamma(t)\Vert\leq\Vert\gamma_n(t)-l_n(t)\Vert+\Vert l_n(t)-\gamma(t)\Vert\nonumber\\
\phantom{}\\
<\big\vert 1-\Vert l_n(t)\Vert_p\big\vert\frac{\Vert l_n(t)\Vert}{\Vert l_n(t)\Vert_p}+\frac{1}{n}\quad\forall\, t\in[-1,1].\nonumber
\end{eqnarray}
Recall that $\gamma\in\Gamma_1$. Since, by \eqref{l2} again,
\begin{eqnarray*}
\max_{t\in[-1,1]}\big\vert 1-\Vert l_n(t)\Vert_p\big\vert= \max_{t\in[-1,1]}\big\vert\Vert\gamma(t)\Vert_p
-\Vert l_n(t)\Vert_p\big\vert\\
\leq\max_{t\in[-1,1]}\Vert\gamma(t)-l_n(t)\Vert_p\leq c\max_{t\in[-1,1]}\Vert\gamma(t)-l_n(t)\Vert\leq\frac{c}{n}
\end{eqnarray*}
for some $c>0$, \eqref{l1} immediately follows from \eqref{l2}--\eqref{l4}.
\end{proof}
Finally, it is known \cite[ Section 4]{Le} that
\begin{itemize}
\item[$({\rm p}_4)$] \textit{$E(\lambda_n)\subseteq C^1(\overline{\Omega})$ for all $n\in\mathbb{N}$.}
\end{itemize}
Let $p:=2$. Through \cite[Proposition 3]{deFG} we also obtain
\begin{itemize}
\item[$({\rm p}_5)$] \textit{If $u$ lies in $E(\lambda_n)$ and vanishes on a set of positive Lebesgue measure then $u=0$.}
\end{itemize}
Setting
$$\bar H_n:=\oplus_{m=1}^n E(\lambda_m),\quad\hat H_n:={\bar H_n}^\perp,$$
each $u\in H^1(\Omega)$ can uniquely be written as $u=\bar u+\hat u$, with $\bar u\in\bar H_n$ and $\hat u\in\hat H_n$, because
 $H^1(\Omega)=\bar H_n\oplus\hat H_n$. By orthogonality one has, for every $n\geq 2$,
\begin{eqnarray}\label{maxmin}
\lambda_n=\max\left\{\frac{\Vert\nabla\bar u\Vert_2^2+\int_{\partial\Omega}\beta\bar u^2 d\sigma}{\Vert\bar u\Vert_2^2}:\bar u\in\bar H_n,\;\bar u\neq 0\right\}\nonumber\\
\phantom{}\\
=\min\left\{\frac{\Vert\nabla\hat u\Vert_2^2+\int_{\partial\Omega}\beta\hat u^2 d\sigma}{\Vert\hat u\Vert_2^2}:\hat u\in\hat H_{n-1},\;\hat u\neq 0\right\}.\nonumber
\end{eqnarray}
A simple argument, based on orthogonality and $({\rm p}_5)$, yields the next result.
\begin{lemma}\label{B}
Let $n\in\mathbb{N}$ and let $\theta\in L^\infty(\Omega)\setminus\{\lambda_n\}$ satisfy $\theta\geq\lambda_n$. Then there exists a constant $\bar c>0$ such that
$$\Vert\nabla\bar u\Vert_2^2+\int_{\partial\Omega}\beta(x)\bar u(x)^2d\sigma-\int_\Omega\theta(x)\bar u(x)^2dx\leq-\bar c\Vert\bar u\Vert^2\quad\forall\,\bar u\in\bar H_n\, .$$
Let $n\in\mathbb{N}$ and let $\theta\in L^\infty(\Omega)\setminus\{\lambda_{n+1}\}$ satisfy $\theta\leq\lambda_{n+1}$. Then there exists a constant $\hat c>0$ such that
$$\Vert\nabla\hat u\Vert_2^2+\int_{\partial\Omega}\beta(x)\hat u(x)^2d\sigma-\int_\Omega\theta(x)\hat u(x)^2dx\geq\hat c\Vert\hat u\Vert^2\quad\forall\,\hat u\in\hat H_n\, .$$
\end{lemma}
\section{Existence results}\label{S3}
To avoid unnecessary technicalities, `for every $x\in\Omega$' will take the place of `for almost every $x\in\Omega$' and the variable $x$ will be omitted when no confusion can arise.

Let $f:\Omega\times\mathbb{R}\to\mathbb{R}$ be a Carath\'eodory function such that $f(\cdot,0)=0$ and let
\begin{equation}\label{defF}
F(x,\xi):=\int_0^\xi f(x,t)dt\, ,\quad (x,\xi)\in\Omega\times\mathbb{R}.
\end{equation}
We will posit the following assumptions.
\begin{itemize}
\item[$({\rm f}_1)$] \textit{To every $\rho>0$ there corresponds $a_\rho\in L^\infty(\Omega)$ satisfying $\displaystyle{\sup_{|t|\leq\rho}}|f(x,t)|\leq a_\rho(x)$ in $\Omega$.}
\item[$({\rm f}_2)$] \textit{$\displaystyle{\limsup_{t\to\pm\infty}\frac{f(x,t)}{|t|^{p-2}t}}\leq a_0<\lambda_1$ uniformly with respect to $x\in\Omega$.}
\item[$({\rm f}_3)$] \textit{There exist $a_1,a_2\in L^\infty(\Omega)\setminus\{\lambda_1\}$ such that $\lambda_1\leq a_1\leq a_2$ and
$$a_1(x)\leq\displaystyle{\liminf_{t\to 0}\frac{f(x,t)}{|t|^{p-2}t}\leq\limsup_{t\to 0}\frac{f(x,t)}{|t^{p-2}t}}\leq a_2(x)\quad \mbox{uniformly in $x\in\Omega$.}$$
}
\item[$({\rm f}_4)$] \textit{To every $\rho>0$ there corresponds $\mu_\rho>0$ such that $t\mapsto f(x,t)+\mu_\rho|t|^{p-2}t$ is nondecreasing on $[-\rho,\rho]$ for all $x\in\Omega$.}
\end{itemize}
\begin{remark}\label{ffour}
Obviously, both $({\rm f}_3)$ and $({\rm f}_4)$ imply that for each $\rho>0$ we can find $\mu_\rho>0$ fulfilling
$$f(x,t)t+\mu_\rho|t|^p\geq 0\quad\forall\, (x,t)\in\Omega\times [-\rho,\rho].$$
\end{remark}
 Now, recall that $X:=W^{1,p}(\Omega)$. The energy functional $\varphi:X\to\R$ stemming from Problem \eqref{prob} is
\begin{equation}\label{defphi}
\varphi(u):=\frac{1}{p}\left(\Vert\nabla u\Vert_p^p+\int_{\partial\Omega}\beta(x)|u(x)|^pd\sigma\right)-\int_\Omega F(x,u(x))\, dx,\quad u\in X,
\end{equation}
with $F$ given by (\ref{defF}). One clearly has $\varphi\in C^1(X)$. Moreover, if $({\rm f}_2)$ holds then, fixed any $\hat a_0\in (a_0,\lambda_1)$, there exists $M>0$ such that
\begin{equation}\label{supsol}
\frac{f(x,t)}{|t|^{p-2}t}<\hat a_0<\lambda_1
\end{equation}
provided $x\in\Omega$ and $|t|\geq M$. Since $({\rm p}_2)$ entails $t_1\hat u_1\geq M$ for $t_1>0$ large enough, inequality
\eqref{supsol} combined with Remark \ref{equiv} lead to
\begin{equation}\label{ss}
\int_\Omega f(x,\hat u)v\, dx\leq\lambda_1\int_\Omega \hat u^{p-1}dx=\langle A_p(\hat u),v\rangle
+\int_{\partial\Omega}\beta\hat u^{p-1}v\, dx,\quad v\in X,\; v\geq 0,
\end{equation}
where $\hat u:=t_1\hat u_1$.
\subsection{Constant-sign solutions}
Define, provided $x\in\Omega$ and $t,\xi\in\R$,
\begin{equation}\label{truncation}
\hat g_+(x,t):=\left\{
\begin{array}{ll}
0 &  \mbox{when $t<0$,}\\
f(x,t)+t^{p-1} & \mbox{if $0\leq t\leq\hat u(x)$,}\\
f(x,\hat t)+\hat t^{p-1} & \mbox{otherwise,}\\
\end{array}
\right.
\end{equation}
$$\hat g_-(x,t):=\left\{
\begin{array}{ll}
f(x,-\hat t)-\hat t^{p-1} &  \mbox{when $t<-\hat u(x)$,}\\
f(x,t)+|t|^{p-2}t & \mbox{if $-\hat u(x)\leq t\leq 0$,}\\
0 & \mbox{otherwise,}\\
\end{array}
\right.$$
as well as
$$\hat  G_\pm(x,\xi):=\int_0^\xi \hat g_\pm(x,t)\, dt.$$
It is evident that the corresponding truncated functionals
$$\hat \psi_\pm(u):=\frac{1}{p}\left(\Vert u\Vert^p+\int_{\partial\Omega}\beta(x)|u(x)|^{p}d\sigma\right)-\int_\Omega \hat G_\pm(x,u(x))\, dx,\quad u\in X,$$
belong to $C^1(X)$ also.
\begin{theorem}\label{css}
Under hypotheses $({\rm f}_1)$--$({\rm f}_3)$,  Problem \eqref{prob} admits at least two constant-sign solutions $u_0\in[0,\hat u]\cap{\rm int}(C_+)$ and $v_0\in[-\hat u,0]\cap(-{\rm int}(C_+))$. If, moreover, $({\rm f}_4)$ holds then $u_0,v_0$ are local minimizers for $\varphi$.
\end{theorem}
\begin{proof}
The space $X$ compactly embeds in $L^p(\Omega)$ while the Nemitskii operator $N_{\hat g_+}$ turns out to be continuous on $L^p(\Omega)$. Thus, a standard argument ensures that $\hat\psi_+$ is weakly sequentially lower semi-continuous. Since, on account of \eqref{truncation}, it is coercive, we have
\begin{equation}\label{defuzero}
\inf_{u\in X}\hat\psi_+(u)=\hat\psi_+(u_0)
\end{equation}
for some $u_0\in X$. Fix $\ep>0$. Assumption  $({\rm f}_3)$ yields $\delta>0$ small such that 
\begin{equation*}
F(x,\xi)\geq\frac{a_1(x)-\ep}{p}|\xi|^p\quad\forall\, (x,\xi)\in\Omega\times[-\delta,\delta].
\end{equation*}
If $\tau\in (0,t_1)$ complies with $\tau\hat u_1\leq\delta$ then, by \eqref{truncation}, the choice of $\tau$, the above inequality, and Remark \ref{equiv},
\begin{eqnarray*}
\hat\psi_+(\tau\hat u_1)\leq\frac{\tau^p}{p}\left(\Vert\nabla\hat u_1\Vert_p^p+\int_{\partial\Omega}\beta\hat u_1^p\, d\sigma -\int_\Omega(a_1-\ep)\hat u_1^p\, dx\right)\\
=\frac{\tau^p}{p}\left(\lambda_1\int_\Omega\hat u_1^p\, dx-\int_\Omega(a_1-\ep)\hat u_1^p\, dx\right)
=\frac{\tau^p}{p}\left(\int_\Omega(\lambda_1-a_1)\hat u^p_1\, dx+\ep\right)<0
\end{eqnarray*}
as soon as $\ep<\int_\Omega(\lambda_1-a_1)\hat u_1^p\, dx$. Hence,
\begin{equation*}
\hat\psi_+(u_0)<0=\hat\psi_+(0),
\end{equation*}
which clearly means $u_0\neq 0$. Now, through (\ref{defuzero}) we get $\hat\psi_+'(u_0)=0$, namely
\begin{equation}\label{gplus}
\langle A_p(u_0)+|u_0|^{p-2}u_0,v\rangle+\int_{\partial\Omega}\beta|u_0|^{p-2}u_0v d\sigma=
\langle N_{\hat g_+}(u_0),v\rangle,\quad v\in X.
\end{equation}
Choosing $v:=-u_0^-$ in \eqref{gplus} leads to $\Vert\nabla u_0^-\Vert_p^p+\Vert u_0^-\Vert_p^p\leq 0$, and $u_0\geq 0$.
Next, pick $v:=(u_0-\hat u)^+$. From \eqref{gplus}, besides \eqref{ss}, it follows
\begin{eqnarray*}
\langle A_p(u_0), (u_0-\hat u)^+\rangle+\int_\Omega u_0^{p-1}(u_0-\hat u)^+dx+\int_{\partial\Omega}\beta u_0^{p-1}(u_0-\hat u)^+d\sigma\\
=\int_\Omega (f(x,\hat u)+\hat u^{p-1})(u_0-\hat u)^+dx\\
\leq\langle A_p(\hat u),(u_0-\hat u)^+\rangle +\int_{\partial\Omega}\beta\hat u^{p-1}(u_0-\hat u)^+d\sigma+\int_\Omega \hat u^{p-1}(u_0-\hat u)^+dx,
\end{eqnarray*}
that is
$$\langle A_p(u_0)-A_p(\hat u), (u_0-\hat u)^+\rangle+\int_{\partial\Omega}\beta(u_0^{p-1}-\hat u^{p-1})(u_0-\hat u)^+d\sigma +\int_\Omega(u_0^{p-1}-\hat u^{p-1})(u_0-\hat u)^+dx\leq 0.$$
Consequently, $u_0\leq\hat u$. Now, \eqref{gplus} becomes
$$\langle A_p(u_0),v\rangle+\int_{\partial\Omega}\beta u_0^{p-1}v\, d\sigma=\int_\Omega f(x,u_0)v\, dx\quad\forall\, v\in X$$
whence, on account of Remark \ref{equiv},
$$-\Delta_p u_0=f(x,u_0)\quad\mbox{in}\quad\Omega,\quad\frac{\partial u}{\partial n_p}+\beta(x)|u_0|^{p-2}u_0=0 \quad\mbox{on}\quad\partial\Omega.$$
Standard regularity arguments ensure that  $u_0\in C_+\setminus\{0\}$. Let $\rho:=\Vert\hat u\Vert_\infty\geq\Vert u_0\Vert_\infty$. Due to Remark \ref{ffour} one has
$$-\Delta_p u_0(x)+\mu_\rho u_0(x)^{p-1}=f(x,u_0(x))+\mu_\rho u_0(x)^{p-1}\geq 0\quad\mbox{a.e. in }\Omega\, .$$
Therefore, by \cite[Theorem 5]{Va}, $u_0\in{\rm int}(C_+)$ and thus $u_0\in [0,\hat u]\cap{\rm int}(C_+)$, as desired. Define $u_\delta:=u_0+\delta$, where $\delta>0$. Since
$$-\Delta_pu_\delta(x)+\mu_\rho u_\delta(x)^{p-1}\leq-\Delta_p u_0(x)+\mu_\rho u_0(x)^{p-1}+o(\delta)=f(x,u_0(x))+\mu_\rho u_0(x)^{p-1}+o(\delta),$$
exploiting $({\rm f}_4)$ and \eqref{supsol} we obtain
\begin{eqnarray*}
-\Delta_p u_\delta(x)+\mu_\rho u_\delta(x)^{p-1}\leq f(x,\hat u(x))+\mu_\rho\hat u(x)^{p-1}+o(\delta)\\
\phantom{}\\
<(\hat a_0+\mu_\rho)\hat u(x)^{p-1}+o(\delta)\leq(\lambda_1+\mu_\rho)\hat u(x)^{p-1}=-\Delta_p\hat u(x)+\mu_\rho\hat u(x)^{p-1}
\end{eqnarray*}
for any $\delta>0$ small enough, because
$$(\lambda_1-\hat a_0)\inf_{x\in\Omega}\hat u(x)^{p-1}>0;$$
cf. \eqref{supsol} as well as $({\rm p}_2)$. Theorem 5 of \cite{Va} gives $u_\delta\leq\hat u$, whence
\begin{equation}\label{intC1}
u_0\in{\rm int}_{C^1(\overline{\Omega})}([0,\hat u]).
\end{equation}
Observe next that $\varphi|_{[0,\hat u]}=\hat\psi_+|_{[0,\hat u]}$ thanks to \eqref{truncation}. So, by \eqref{intC1} and \eqref{defuzero}, the function $u_0$ is a $C^1(\overline{\Omega})$-local minimizer for $\varphi$. Finally, \cite[Proposition 3]{PaRa} guarantees that the same holds putting $X$ in place of $C^1(\overline{\Omega})$.\\
A similar argument produces $v_0\in[-\hat u,0]\cap(-{\rm int}(C_+))$ with the asserted properties. 
\end{proof}
\begin{remark}
The upper bound at zero requested by $({\rm f}_3)$ for $t\mapsto f(x,t)/|t|^{p-2}t$ has not been used to find constant-sign solutions.
\end{remark}
The next result looks like \cite[Theorem 3.3]{MaPaTMNA}; see also \cite[Proposition 8]{PaRa}. So, we will only sketch its proof.
\begin{theorem}\label{extremal}

Let $({\rm f}_1)$--$({\rm f}_3)$ be satisfied. Then \eqref{prob} possesses the smallest (resp., the biggest) nontrivial solution $u_*$ in $[0,\hat u]$ (resp., $v_*$ in $[-\hat u,0]$). Further, $-v_*,u_*\in{\rm int}(C_+)$.
\end{theorem}
\begin{proof}
Define $\Sigma_+:=\{u\in X\setminus\{0\}:\mbox{$u$ solves \eqref{prob} and $0\leq u\leq\hat u$}\}$. Due to Theorem \ref{css} one has $\Sigma_+\neq\emptyset$. Actually, $\Sigma_+\subseteq{\rm int}(C_+)$. The same arguments employed in establishing \cite[Proposition 8]{APS1} show here that
\begin{enumerate}
\item[1)] $\Sigma_+$ is downward directed, and
\item[2)] $\inf\Sigma_+=\displaystyle{\inf_{n\in\mathbb{N}}}u_n=u_*$ for some $\{u_n\}\subseteq\Sigma_+$, $u_*\in X$ fulfilling
$u_n\to u_*$ in $X$ and $u_n(x)\to u_*(x)$ a.e. in $\Omega$.
\end{enumerate}
Hence, $u_*$ turns out to be a solution of \eqref{prob} lying in $[0,\hat u]$. It remains to verify that $u_*\neq 0$. Suppose on the contrary $u_*=0$. Reasoning exactly as in the proof of \cite[Proposition 14]{APS2} we obtain $\alpha\in L^\infty(\Omega)$ and $w\in{\rm int}(C_+)$ with the properties below.
\begin{enumerate}
\item[3)] $a_1\leq\alpha\leq a_2$.
\item[4)] $-\Delta_p w(x)=\alpha(x)|w(x)|^{p-2}w(x)$ a.e. in $\Omega$, $\displaystyle{\frac{\partial w}{\partial n_p}}+ \beta(x)|w|^{p-2}w=0$ on $\partial\Omega$.
\end{enumerate}
Let $v\in{\rm int}(C_+)$.  Gathering \cite[Theorem 1.1]{AH} and \cite[Theorem 2.4.54]{GaPaNA} together produce
\begin{eqnarray*}
0\leq\int_\Omega\left(\vert\nabla v\vert^p-\nabla(\frac{v^p}{w^{p-1}})\cdot\vert\nabla w\vert^{p-2}\nabla w\right) dx\\
=\int_\Omega\left(\vert\nabla v\vert^p-\frac{v^p}{w^{p-1}}(-\Delta_p w)\right) dx-\int_{\partial\Omega}\frac{\partial u}{\partial n_p}\, \frac{v^p}{w^{p-1}}\, d\sigma\\
=\int_\Omega\left(\vert\nabla v\vert^p-\alpha v^p\right) dx+\int_{\partial\Omega}\beta v^p d\sigma,
\end{eqnarray*}
where 4) has been used. If $v=\hat u_1$ then, by 3) and $({\rm f}_3)$,
$$0\leq\int_\Omega(\lambda_1-\alpha(x))\hat u_1(x)^pdx\leq\int_\Omega(\lambda_1-a_1(x))\hat u_1(x)^pdx<0,$$
which is impossible. Therefore, $u_*\in\Sigma_+$, and the conclusion follows. A similar argument applies to get $v_*$. 
\end{proof}
\subsection{Nodal solutions}
Define, for every $x\in\Omega$ and $t,\xi\in\R$,
\begin{equation}\label{truncf}
\hat f(x,t):=\left\{
\begin{array}{ll}
f(x,v_*(x))+|v_*(x)|^{p-2}v_*(x) &  \mbox{when $t<v_*(x)$,}\\
f(x,t)+|t|^{p-2}t & \mbox{if $v_*(x)\leq t\leq u_*(x)$,}\\
f(x,u_*(x))+u_*(x)^{p-1} & \mbox{when $t>u_*(x)$,}\\
\end{array}
\right.
\end{equation}
$$\hat f_\pm(x,t):=\hat f(x,t^\pm),$$
as well as
$$\hat F(x,\xi):=\int_0^\xi \hat f(x,t)dt,\quad \hat  F_{\pm}(x,\xi):=\int_0^\xi\hat  f_\pm(x,t)\, dt.$$
It is evident that the corresponding truncated functionals
$$\hat\varphi(u):=\frac{1}{p}\left(\Vert u\Vert^p+\int_{\partial\Omega}\beta(x)|u(x)|^{p-1}d\sigma\right)-\int_\Omega\hat F(x,u(x))\, dx,\quad u\in X,$$
$$\hat\varphi_{\pm}(u):=\frac{1}{p}\left(\Vert u\Vert^p+\int_{\partial\Omega}\beta(x)|u(x)|^{p-1}d\sigma\right)-\int_\Omega\hat F_\pm(x,u(x))\, dx,\quad u\in X,$$
belong to $C^1(X)$. 
\begin{lemma}\label{critset}
Under hypotheses $({\rm f}_1)$--$({\rm f}_3)$ one has
$$K(\hat\varphi)\subseteq [v_*,u_*],\quad K(\hat\varphi_-)=\{0,v_*\},\quad K(\hat\varphi_+)=\{0,u_*\}.$$
\end{lemma}
\begin{proof}
If $u\in K(\hat\varphi)$ then
$$\langle A_p(u)+|u|^{p-2}u,v\rangle+\int_{\partial\Omega}\beta|u|^{p-2}uv\, d\sigma
=\langle N_{\hat f}(u),v\rangle\quad\forall\, v\in X.$$
Letting $v:=(u-u_*)^+$ yields
\begin{eqnarray*}
\langle A_p(u), (u-u_*)^+\rangle+\int_\Omega u^{p-1}(u-u_*)^+dx+\int_{\partial\Omega}\beta|u|^{p-1}(u-u_*)^+d\sigma\\
=\int_\Omega (f(x,u_*)+u_*^{p-1})(u-u_*)^+dx.
\end{eqnarray*}
Since, by Theorem \ref{extremal}, the function $u_*$ solves \eqref{prob}, this results in
$$\langle A_p(u)-A_p(u_*),(u-u_*)^+\rangle+\int_\Omega(u^{p-1}-u_*^{p-1})(u-u_*)^+dx
+\int_{\partial\Omega}\beta(u^{p-1}-u_*^{p-1})(u-u_*)^+ d\sigma=0.$$
Therefore, $m(\{ x\in\Omega: u(x)>u_*(x)\})=0$, whence $u\leq u_*$. An analogous reasoning provides $u\geq v_*$, and the first inclusion holds.

As before, we obtain $K(\hat\varphi_-)\subseteq[v_*,0]$, while the extremality of $v_*$ (see Theorem \ref{extremal}) forces
$K(\hat\varphi_-)=\{v_*,0\}$. The remaining proof is similar.
\end{proof}
\begin{lemma}\label{minimizer}
Let $({\rm f}_1)$--$({\rm f}_4)$ be satisfied. Then $u_*,v_*$ are local minimizers for $\hat\varphi$.
\end{lemma}
\begin{proof}
The space $X$ compactly embeds in $L^p(\Omega)$ while the Nemitskii operator $N_{\hat f_+}$ turns out to be continuous on $L^p(\Omega)$. Thus, a standard argument ensures that $\hat\varphi_+$ is weakly sequentially lower semi-continuous. Since, on account of \eqref{truncf}, it is coercive, we have
\begin{equation}\label{defhatuzero}
\inf_{u\in X}\hat\varphi_+(u)=\hat\varphi_+(u_0)
\end{equation}
for some $u_0\in X$. Reasoning as in the proof of Theorem \ref{css} produces $\hat\varphi_+(u_0)<0$, i.e., $u_0\neq 0$. Hence, by \eqref{defhatuzero} and Lemma \ref{critset}, $u_0=u_*\in{\rm int}(C_+)$. Since $\hat\varphi|_{C_+}=\hat\varphi_+|_{C_+}$, the function $u_0$ turns out to be a $C^1(\overline{\Omega})$-local minimizer for $\hat\varphi$.  Now, Proposition 3 in \cite{PaRa} guarantees that the same holds true with $X$ in place of $C^1(\overline{\Omega})$. A similar argument applies to $v_*$.
\end{proof}
\begin{theorem}\label{furthersol}
Under hypotheses  $({\rm f}_1)$--$({\rm f}_4)$, with $\displaystyle{\essinf_{x\in\Omega}}\, a_1(x)>\lambda_2$, Problem \eqref{prob} possesses a nodal solution $u_1\in [v_*,u_*]\cap C^1(\overline{\Omega})$.
\end{theorem}
\begin{proof}
By Theorem \ref{extremal} and Lemma \ref{critset}, we may assume $K(\hat\varphi)$ finite. Let $\hat\varphi(v_*)\leq\hat \varphi(u_*)$ (the other case is analogous). Without loss of generality, the local minimizer $u_*$ for $\hat\varphi$ (cf. Lemma \ref{extremal}) can be supposed proper. Thus, there exists $\rho\in(0,\Vert u_*-v_*\Vert)$ such that
\begin{equation}\label{crho}
\hat\varphi(u_*)<c_\rho:=\inf_{u\in\partial B_\rho(u_*)}\hat\varphi(u).
\end{equation}
Moreover, $\hat\varphi$ fulfils Condition (PS) because, due to \eqref{truncf}, it is coercive; see Proposition \ref{compactness}. So, the Mountain Pass Theorem yields a point $u_1\in X$ complying with $\hat\varphi'(u_1)=0$ and
\begin{equation}\label{defw}
c_\rho\leq\hat\varphi(u_1)= \inf_{\gamma\in\Gamma}\max_{t\in[0,1]}\hat\varphi(\gamma(t)),
\end{equation}
where
$$\Gamma:=\{\gamma\in C^0([0,1],X):\;\gamma(0)=v_*,\;\gamma(1)=u_*\}\, .$$
Obviously, $u_1$ solves (\ref{prob}). Through (\ref{crho})--(\ref{defw}), besides Lemma \ref{critset}, we get
$$u_1\in[v_*,u_*]\setminus\{v_*,u_*\},$$
 while standard regularity arguments yield $u_1\in C^1(\overline{\Omega})$. The proof is thus completed once one verifies that $u_1\neq 0$. This will follow from the inequality
\begin{equation}\label{wminore}
\hat\varphi(u_1)<0\, ,
\end{equation}
which, in view of (\ref{defw}), can be shown by constructing  a path $\tilde\gamma \in\Gamma$ such that
\begin{equation}\label{path}
\hat\varphi(\tilde\gamma(t))<0\quad\forall\, t\in [0,1]\, .
\end{equation}
By (${\rm f}_3$) to every $\eta>0$ there corresponds $\delta>0$ such that
\begin{equation}\label{small} 
F(x,\xi)\geq\frac{a_1(x)-\eta}{p}|\xi|^p,\quad (x,z)\in\Omega\times[-\delta,\delta].
\end{equation}
Combining (${\rm p}_4$) with Lemma \ref{density} entails
\begin{equation}\label{smalleta}
\max_{t\in [-1,1]}\Phi(\gamma_\eta(t))<\lambda_2+\eta
\end{equation}
for appropriate $\gamma_\eta\in\Gamma_C$. Since $\gamma_\eta([-1,1])$ is compact in $C^1(\overline{\Omega})$ and $-v_*,u_* \in{\rm int}(C_+)$ we can find $\ep>0$ so small that
$$v_*(x)\leq\ep\gamma_\eta(t)(x)\leq u_*(x),\quad|\ep\gamma_\eta(t)(x)|
\leq\delta$$
whenever $x\in\Omega$, $t\in[-1,1]$. Thanks to \eqref{small}--\eqref{smalleta} one has
\begin{eqnarray*}
\hat\varphi(\ep\gamma_\eta(t))=\frac{\ep^p}{p}\left(\Vert\gamma_\eta(t)\Vert^p+\int_{\partial\Omega}\beta|\gamma_\eta(t)|^pd\sigma\right)-\int_\Omega \hat F(x,\ep\gamma_\eta(t)(x)) dx\\
\phantom{}\\
<\frac{\ep^p}{p}\left(\Phi(\gamma_\eta(t))+\int_{\Omega}(\eta-a_1)|\gamma_\eta(t)|^pdx\right)
<\frac{\ep^p}{p}\left(\lambda_2+2\eta-\essinf_{x\in\Omega} a_1(x)\right)<0
\end{eqnarray*}
provided $\eta<\frac{1}{2}(\displaystyle{\essinf_{x\in\Omega}} a_1(x)-\lambda_2)$, because $\gamma_\eta(t)\in U_C$. Consequently,
\begin{equation}\label{middle}
\hat\varphi|_{\ep\gamma_\eta([-1,1])}<0.
\end{equation}
Next, write $a:=\hat\varphi_+(u_*)$. From the proof of Lemma \ref{minimizer} it follows $a<0$. We
may suppose
$$K(\hat\varphi_+)=\{0,u_*\},$$
otherwise the conclusion is straightforward. Hence, no critical value of $\hat\varphi_+$ lies
in $(a,0)$ while
$$K_a(\hat\varphi_+) =\{ u_*\}\, .$$
Due to the second deformation lemma \cite[Theorem 5.1.33]{GaPaNA}, there exists a continuous function $h:[0,1]\times (\hat\varphi_+^0\setminus\{0\})\to\hat\varphi_+^0$ satisfying
$$h(0,u)=u\, ,\quad h(1,u)=u_*\, ,\quad\mbox{and}\quad\hat\varphi_+(h(t,u))\leq\hat\varphi_+(u)$$
for all $(t,u)\in [0,1]\times(\hat\varphi_+^0\setminus\{0\})$. Let $\gamma_+(t):=h(t,\ep \hat u_0)^+$, $t\in [0,1]$. Then $\gamma_+(0)=\ep\hat u_0$, $\gamma_+(1)=u_*$, as well as
\begin{equation}\label{gammapiu}
\hat\varphi(\gamma_+(t))=\hat\varphi_+(\gamma_+(t))\leq\hat\varphi_+(h(t,\ep\hat u_0)) \leq\hat\varphi_+(\ep\hat u_0)=\hat\varphi(\ep\gamma_\eta(1))<0;
\end{equation}
cf. \eqref{middle}. In a similar way, but with $\hat\varphi_+$ replaced by $\hat\varphi_-$, we can construct a continuous function
$\gamma_-:[0,1]\to X$ such that $\gamma_-(0)=v_*$, $\gamma_-(1)=-\ep\hat u_0$, and
\begin{equation}\label{gammameno}
\varphi(\gamma_-(t))<0\quad\forall\, t\in [0,1].
\end{equation}
Concatenating $\gamma_-$, $\ep\gamma_\eta$, and $\gamma_+$ one obtains a path $\tilde\gamma \in\Gamma$ which, in view of (\ref{middle})--(\ref{gammameno}), fulfils (\ref{path}).
\end{proof}
The next multiplicity result directly stems from Theorems \ref{css}--\ref{furthersol}.
\begin{theorem}\label{concl}
Let $({\rm f}_1)$--$({\rm f}_4)$ be satisfied and let $\displaystyle{\essinf_{x\in\Omega}}\, a_1(x)>\lambda_2$. Then \eqref{prob} admits at least three nontrivial solutions: $u_0\in{\rm int}(C_+)$, $v_0\in-{\rm int}(C_+)$, and $u_1\in[v_0,u_0]\cap C^1(\overline{\Omega})$ nodal.
\end{theorem}
An immediate application of this result produces both constant-sign and nodal solutions to the problem
\begin{equation}\label{sc1}
\left\{
\begin{array}{ll}
-\Delta_p u=\lambda |u|^{p-2}u-g(x,u) & \mbox{in }\Omega,\\
\displaystyle{\frac{\partial u}{\partial n_p}}+\beta(x)|u|^{p-2}u=0 & \mbox{on } \partial\Omega,\\
\end{array}
\right.
\end{equation}
where $\lambda>0$ while $g:\Omega\times\mathbb{R}\to\mathbb{R}$ denotes a Carath\'eodory function such that $g(\cdot,0)=0$. Under Dirichlet boundary conditions, the above equation has been widely investigated; see for instance \cite{PaPa,AMM,MoTa} and the references given there.
\begin{theorem}\label{firstcons}
Assume that $\lambda>\lambda_2$. If, moreover,
\begin{itemize}
\item[$({\rm g}_1)$] to every $\rho>0$ there corresponds $b_\rho\in L^\infty(\Omega)$ satisfying $\displaystyle{\sup_{|t|\leq\rho}}|g(x,t)|\leq b_\rho(x)$ in $\Omega$,
\item[$({\rm g}_2)$] $\displaystyle{\liminf_{t\to\pm\infty}\frac{g(x,t)}{|t|^{p-2}t}}\geq b_0>\lambda-\lambda_1$ uniformly with respect to $x\in\Omega$,
\item[$({\rm g}_3)$] $b_1\leq\displaystyle{\liminf_{t\to 0}\frac{g(x,t)}{|t|^{p-2}t}}\leq\displaystyle{\limsup_{t\to 0}
\frac{g(x,t)}{|t|^{p-2}t}}\leq b_2<\lambda-\lambda_1$ uniformly in $x\in\Omega$, as well as
\item[$({\rm g}_4)$] for every $\rho>0$ there exists $\mu_\rho>\lambda$ such that $t\mapsto\mu_\rho|t|^{p-2}t-g(x,t)$ is nondecreasing on $[-\rho,\rho]$ whatever $x\in\Omega$,
\end{itemize}
then \eqref{sc1} possesses at least three nontrivial solutions: $u_0\in{\rm int}(C_+)$, $v_0\in-{\rm int}(C_+)$, and $u_1\in[v_0,u_0] \cap C^1(\overline{\Omega})$ nodal.
\end{theorem}
Conditions $({\rm g}_2)$--$({\rm g}_3)$ above are much more general than the corresponding ones of \cite[Theorem 12]{PaRa} but $({\rm g}_4)$ does not appear in that result. A similar comment holds true for \cite[Theorem 3.1]{GaPaCPAA}, where $\beta\equiv 0$ and sub-critical behavior for $t\mapsto g(x,t)$ is taken on. Finally, the $\beta\equiv 0$ version of Theorem \ref{firstcons} and \cite[Theorem 4.1]{BaCaMo} are mutually independent.
\section{The semilinear case} 
Suppose $f:\Omega\times\R\to\R$ is a function such that $f(\cdot,0)=0$ and $f(x,\cdot)$ belongs to $C^1(\R)$ for every $x\in\Omega$, while $f(\cdot,t)$ and $f'_t(\cdot,t)$ are measurable for all $t\in\mathbb{R}$. Let $F$ be given by \eqref{defF}. We will make the following assumptions.
\begin{itemize}
\item[$({\rm f}_5)$] \textit{$|f'_t(x,t)|\leq a_3(1+|t|^{r-2})$ in $\Omega\times\mathbb{R}$, where $2\leq r<2^*$.}
\item[$({\rm f}_6)$] \textit{$f'_t(x,0)=\displaystyle{\lim_{t\to 0}\frac{f(x,t)}{t}}$ uniformly with respect to $x\in\Omega$. Moreover, there exists $m\geq 2$ such that $\lambda_m\leq f'_t(\cdot,0)\leq\lambda_{m+1}$, $f'_t(\cdot,0)\neq\lambda_m$, and 
$$F(x,\xi)\leq\displaystyle{\frac{\lambda_{m+1}}{2}}\xi^2\quad\forall\,(x,\xi)\in \Omega\times\R.$$
}
\item[$({\rm f}_7)$] \textit{If $m=2$ then $\lambda_2<a_4\leq f'_t(x,0)$ uniformly in $x\in\Omega$.}  
\end{itemize}
It should be noted that $({\rm f}_5)$ implies both $({\rm f}_1)$ and $({\rm f}_4)$ written for $p=2$, while  $({\rm f}_6)$ forces  $({\rm f}_3)$ with $p=2$.  Consider the semi-linear problem
\begin{equation}\label{probtwo}
\left\{
\begin{array}{ll}
-\Delta u=f(x,u) & \mbox{ in }\Omega\, ,\\
\displaystyle{\frac{\partial u}{\partial n}}+\beta(x)u=0 & \mbox{ on }\partial\Omega\, ,
\end{array}
\right.
\end{equation}
where $\frac{\partial u}{\partial n}:=\nabla u\cdot n$; see \cite[Remark 1.40]{MoMoPa}. If $X:=H^1(\Omega)$ then the energy functional $\varphi:X\to\R$ stemming from \eqref{probtwo} is
\begin{equation}\label{phitwo}
\varphi(u):=\frac{1}{2}\left(\Vert\nabla u\Vert^2_2+\int_{\partial\Omega}\beta(x)u(x)^2d\sigma\right)-\int_\Omega F(x,u(x))\, dx,\quad u\in X.
\end{equation}
Obviously, $\varphi\in C^2(X)$ and one has
\begin{equation}\label{Fred}
\langle\varphi''(u)(v),w\rangle=\int_\Omega\nabla v\cdot\nabla w\, dx+\int_{\partial\Omega}\beta vw\, d\sigma-\int_\Omega f'_t(x,u)vw\, dx\quad\forall\, u,v,w\in X.
\end{equation}
\begin{lemma}\label{index}
Let $({\rm f}_5)$--$({\rm f}_6)$ be satisfied. Then $C_q(\varphi,0)=\delta_{q,d_m}\mathbb{Z}$ for all $q\in\mathbb{N}_0$, where $d_m:={\rm dim}(\bar H_m)$.
\end{lemma}
\begin{proof} Suppose $f'_t(\cdot,0)\neq\lambda_{m+1}$. By $({\rm f}_6)$, Lemma \ref{B} can be applied with $\theta(x):= f'_t(x,0)$. So, $u=0$ is a non-degenerate critical point of $\varphi$ having Morse index $d_m$, and the conclusion follows from \eqref{kd}. Let now $f'_t(\cdot,0)=\lambda_{m+1}$. Thanks to $({\rm f}_6)$ again, for every $\ep>0$ there exists $\delta>0$ such that
\begin{equation}\label{lim}
F(x,\xi)\geq\frac{1}{2}(f'_t(x,0)-\ep)|\xi|^2\quad\mbox{in}\quad\Omega\times[-\delta,\delta].
\end{equation}
Since $\bar H_m$ is finite dimensional, we can find $\rho>0$ fulfilling
$$u\in\bar H_m\cap B_\rho(0)\implies |u(x)|\leq\delta\quad\forall\, x\in\Omega.$$
Through \eqref{lim} and Lemma \ref{B} this entails
\begin{eqnarray*}
\varphi(u)\leq\frac{1}{2}\left(\Vert\nabla u\Vert_2^2+\int_{\partial\Omega}\beta u^2d\sigma -\int_\Omega f'_t(x,0) u^2dx\right)+\frac{\ep}{2}\Vert u\Vert^2\\
\leq\frac{1}{2}(-\bar c+\ep)\Vert u\Vert^2\leq 0,\quad u\in\bar H_m\cap B_\rho(0),
\end{eqnarray*} 
whenever $\ep<\bar c$. Combining $({\rm f}_6)$ with \eqref{maxmin} we obtain
$$\varphi(u)\geq\frac{1}{2}\left(\Vert\nabla u\Vert_2^2+\int_{\partial\Omega}\beta u^2d\sigma-\lambda_{m+1}\Vert u\Vert_2^2\right)\geq 0,\quad u\in\hat H_m\cap  B\rho(0).$$
Now, \cite[Proposition 2.3]{Su} directly yields the conclusion.
\end{proof}
\begin{theorem}\label{foursol}
Under assumptions $({\rm f}_2)$ and $({\rm f}_5)$--$({\rm f}_7)$, Problem \eqref{probtwo} admits at least four solutions: $u_0\in{\rm int}(C_+)$, $v_0\in-{\rm int}(C_+)$, and $u_1,v_1\in{\rm int}_{C^1(\overline{\Omega})}([v_0,u_0])$ nodal. 
\end{theorem}
\begin{proof}
The existence of $u_0$, $v_0$, $u_1$ comes from Theorem \ref{concl}. Bearing in mind Theorem \ref{extremal} and
Lemma \ref{minimizer}, we may suppose $u_0$, $v_0$ extremal constant-sign solutions to \eqref{probtwo}, i.e., $u_0=u_*$, $v_0=v_*$, as well as local minimizers for $\hat\varphi$. Thus,
\begin{equation}\label{morseone}
C_q(\hat\varphi,u_0)=C_q(\hat\varphi,v_0)=\delta_{q,0}\mathbb{Z}\quad\forall\, q\in\mathbb{N}_0;
\end{equation}
see \cite[Example 6.45]{MoMoPa}. Let us next verify that $u_1\in{\rm int}_{C^1(\overline{\Omega})}([v_0,u_0])$. Put
$$\rho:=\max\{\Vert u_0\Vert_\infty,\Vert v_0\Vert_\infty\}.$$
If $\mu_\rho$ is as in $({\rm f}_4)$ then
$$-\Delta(u_0-u_1)+\mu_\rho(u_0-u_1)=[f(x,u_0)+\mu_\rho u_0]-[f(x,u_1)+\mu_\rho u_1]\geq 0$$
because $u_1\leq u_0$. So, by \cite[Theorem 5]{Va}, $u_0-u_1\in{\rm int}(C_+)$. Likewise, $u_1-v_0\in{\rm int}(C_+)$, and the assertion follows.\\
The proof of Theorem \ref{furthersol} ensures that $u_1$ is a Mountain Pass type critical point for $\hat\varphi$. Thanks to \eqref{Fred}, Corollary 6.102 in \cite{MoMoPa} gives
\begin{equation}\label{morsetwo}
C_q(\hat\varphi,u_1)=\delta_{q,1}\mathbb{Z}\, ,\quad q\in\mathbb{N}_0.
\end{equation}
From $\hat\varphi|_{[v_0,u_0]}=\varphi|_{[v_0,u_0]}$ and Lemma \ref{index} we infer
\begin{equation}\label{morsethree}
C_q(\hat\varphi,0)=\delta_{q,d_m}\mathbb{Z}\quad\forall\, q\in\mathbb{N}_0,
\end{equation}
while the coercivity of $\hat\varphi$ entails (cf. \cite[Proposition 6.64]{MoMoPa})
\begin{equation}\label{morsefour}
C_q(\hat\varphi,\infty)=\delta_{q,0}\mathbb{Z}\, ,\quad q\in\mathbb{N}_0.
\end{equation}
Now, if $K(\hat\varphi)=\{0,u_0,v_0,u_1\}$ then the Morse relation \eqref{morse} written for $t=-1$  and \eqref{morseone}--\eqref{morsefour} would imply
$$(-1)^{d_m}+2(-1)^0+(-1)^1=(-1)^0,$$
which is impossible. Thus, there exists a further point $v_1\in K(\hat\varphi)\setminus\{0,u_0,v_0,u_1\}$. Lemma \ref{critset}, combined with \eqref{truncf}, shows that $v_1$ turns out to be a nodal solution of \eqref{probtwo} that lies in $[v_0,u_0]$. Standard regularity arguments provide $v_1\in C^1(\overline{\Omega})$. Finally, reasoning as before one achieves
$v_1\in{\rm int}_{C^1(\overline{\Omega})}([v_0,u_0])$.
\end{proof}
An immediate application of this result produces both constant-sign and nodal solutions to the problem
\begin{equation}\label{probthree}
\left\{
\begin{array}{ll}
-\Delta u=\lambda u-g(x,u) & \mbox{in }\Omega,\\
\displaystyle{\frac{\partial u}{\partial n}}+\beta(x)u=0 & \mbox{on } \partial\Omega,\\
\end{array}
\right.
\end{equation}
where $\lambda>0$ while $g:\Omega\times\mathbb{R}\to\mathbb{R}$ denotes a function such that $g(\cdot,0)=0$ and $g(x,\cdot)$ belongs to $C^1(\R)$ for every $x\in\Omega$, while $g(\cdot,t)$ and $g'_t(\cdot,t)$ are measurable for all $t\in\mathbb{R}$. 
\begin{theorem}\label{secondcons}
Let $\lambda\in (\lambda_m,\lambda_{m+1}]$ for some $m\geq 2$ and let $({\rm g}_2)$ of Theorem \ref{firstcons} be satisfied with $p=2$.  If, moreover,
\begin{itemize}
\item[$({\rm g}_5)$] \textit{$|g'_t(x,t)|\leq b_3(1+|t|^{r-2})$ in $\Omega\times\mathbb{R}$, where $2\leq r<2^*$,}
\item[$({\rm g}_6)$] $\displaystyle{g'_t(x,0)=\lim_{t\to 0}\frac{g(x,t)}{t}}=0$ uniformly with respect to $x\in\Omega$, and
\item[$({\rm g}_7)$] $\displaystyle{\int_0^\xi g(x,t)\, dt\geq\frac{\lambda-\lambda_{m+1}}{2}}\xi^2$ for all $(x,\xi)\in\Omega\times\mathbb{R}$,
\end{itemize}
then the same conclusion of Theorem \ref{foursol} holds true concerning \eqref{probthree}.
\end{theorem}
The sign condition $tg(x,t)\geq 0$, $(x,t)\in\Omega\times\mathbb{R}$, clearly forces $({\rm g}_8)$. So, Theorem \ref{secondcons} basically extends \cite[Theorem 14]{PaRa}. For $\beta\equiv 0$, cf. also \cite[Theorem 3.7]{DaMaPa}, \cite[Section 4]{GaPaCPAA}, and the references given there.
\section*{Acknowledgement}
Work performed under the auspices of GNAMPA of INDAM.
\end{document}